\documentclass[12pt, english, reqno]{article}

\usepackage[english]{babel}

\usepackage[letterpaper,top=2cm,bottom=2cm,left=3cm,right=3cm,marginparwidth=1.75cm]{geometry}

\providecommand{\keywords}[1]
{
  \small	
  \textbf{\textit{Keywords---}} #1
}

\usepackage{booktabs, latexsym}
\usepackage{amsmath}
\usepackage{amsthm}
\usepackage{amssymb}
\usepackage{graphicx}
\usepackage{bbm}
\usepackage[colorlinks=true, allcolors=blue]{hyperref}

\numberwithin{equation}{section}

\newtheorem{theorem}{Theorem}[section]
\newtheorem{lemma}[theorem]{Lemma}
\newtheorem{proposition}[theorem]{Proposition}

\theoremstyle{definition}
\newtheorem{definition}[theorem]{Definition}

\theoremstyle{remark}
\newtheorem{remark}[theorem]{Remark}

\newcommand{\ol}{\overline}
\newcommand{\wt}{\widetilde}
\newcommand{\wh}{\widehat}
\newcommand{\mc}{\mathcal}

\newcommand{\N}{\mathbb N}
\newcommand{\Z}{\mathbb Z}
\newcommand{\E}{\mathcal E}
\newcommand{\B}{\mathcal B}

\newcommand{\A}{\mathcal A}

\newcommand{\be}{\begin{equation}}
\newcommand{\ee}{\end{equation}}
\newcommand{\ba}{\begin{aligned}}
\newcommand{\ea}{\end{aligned}}

\begin{document}		

\title{IFS measures on generalized Bratteli diagrams}
\author{Sergey Bezuglyi and Palle E.T. Jorgensen}

\date{}
\maketitle

\begin{center}
    \textit{Dedicated to the memory of Robert Strichartz}
\end{center}

\begin{abstract}
The purpose of the paper is a general analysis of path space 
 measures. Our focus is a certain path space analysis 
 on generalized Bratteli diagrams. We use this in a 
systematic study of systems of self-similar measures (the 
term ``IFS measures'' is used in the paper) for both types of such 
diagrams, discrete and continuous. In special 
 cases,  such measures arise in the study of iterated function 
 systems (IFS). In the literature, similarity may be defined by, 
 e.g., systems of affine maps (Sierpinski), or systems of conformal 
 maps (Julia). We study new classes of semi-branching 
 function systems related to stationary Bratteli diagrams. The 
 latter plays a big role in our understanding 
 of new forms of harmonic analysis on fractals. The measures 
 considered here arise in classes of discrete-time, multi-level 
 dynamical systems where similarity is specified between levels. 
 These structures are made precise by prescribed systems of 
 functions which in turn serve to define self-similarity, i.e., 
the similarity of large scales, and  small scales. 
 For path space systems,
in our main result, we give a necessary and sufficient condition 
for the existence of such generalized IFS measures. For the 
corresponding semi-branching function systems, we further identify 
the measures which are also shift-invariant.

\end{abstract}


\tableofcontents

\section{Introduction} \label{sect_intro}

We present new results in the study of self-similar measures on
path space for infinite graph models,here called generalized 
Bratteli diagrams. By this, we mean particular
graph systems $B$ with the property that the corresponding sets $V$
of vertices, and $E$ of edges, admit discrete level-structures, 
defined below. This means that $V$ is a disjoint union of the
sets $V_n$
and $E$ is a disjoint union of the sets $E_n$. We emphasize that 
the class of generalized Bratteli diagrams 
include two cases: (1) all the sets $V_n, E_n$ are countable (if all
$V_n$'s are finite, then we have a standard Bratteli diagram); 
(2) all the set $V_n, E_n$ are standard Borel spaces. In case (2),
we will say that $B$ is a measurable (or continuous) Bratteli 
diagram. 

Our present framework is motivated by, but more versatile than, 
the more familiar setting
in earlier studies of Bratteli diagrams. We use the discrete levels in order to identify a
class of self-similar path space measures, called iterated function system (IFS) measures.
The latter in turn are inspired by earlier studies 
of IFS systems arising the analysis of fractals, such as various 
Sierpinski constructs, or conformal attractors. However, by 
contrast our 
analysis centers around a new path space analysis for 
discrete-time 
random walk models in generalized Bratteli diagrams. We also 
discuss IFS measures for the case when the system of levels making 
up $B$ are instead standard measure spaces, so non-discrete.

We recall that a measure $\mu$ on a standard Borel space $X$ is 
called \textit{self-similar} with respect to an iterated function 
system $f_i : X \to X, i \in I,$ if
$$
\mu = \sum_{i \in I} p_i \mu\circ f_i^{-1},
$$
where $p = (p_i)_{i\in I}$ is a probability vector and 
$I$ is finite or countable.

A main aim of our paper is an identification of, and an analysis of,
iterated function system (IFS) measures on the path space $X_B$ 
of generalized
Bratteli diagrams. This entails two tasks, (a) an analysis of IFS 
structure of the path spaces, and (b) a study of the particular 
IFS measures on them. 

Our main results consist of finding an explicit construction of
IFS measures for both discrete and measurable generalized 
Bratteli diagrams.

In order to motivate, and to place this in context, we add the 
following three comments here: (i) Previously, IFS measures have 
been considered in special geometries of self-similar fractals (see 
e.g., \cite{Jorgensen2018(book)} and the papers cited there). These
standard self-similar fractals may be realized in finite-dimensional
Euclidean space. By contrast, there are no previous studies of 
IFS measures in path space, i.e., IFS measures realized in the class
of path space structures considered here, standard and 
generalized Bratteli 
diagrams. (ii) While there are earlier results on other, different 
but related, classes of measures on path space of generalized 
Bratteli diagrams, e.g., tail invariant measures, and Markov 
measures, our present identification of IFS measures on path space 
is new. (iii) In order to prepare the reader for the IFS path space 
measures, it will be necessary for us to begin with an account of 
tail invariant measures, and Markov measures. The tools involved 
there for generalized Bratteli diagrams are also needed in our 
introduction of the new IFS measures. But this means our main 
results for IFS measures will be postponed to sections 
\ref{s:discrete GBD} and \ref{sect MBD} below, after
the necessary preliminaries have been presented.  

Generalized Bratteli diagrams considered here arise in various 
areas and have many applications. We mention Cantor and Borel 
dynamics where they are used to construct models of 
transformations, see \cite{HermanPutnamSkau1992}, 
\cite{GiordanoPutnamSkau1995},  \cite{Durand2010}, 
\cite{BezuglyiKarpel2016}, \cite{BezuglyiDooleyKwiatkowski2006}.
Measurable Bratteli diagrams can be met in the theory of Markov
chains, we refer to \cite{Nummelin_1984} as an example of such 
literature.  Several other 
applications come to mind, 
(1) models from financial mathematics, and (2) neural networks. In 
each case, the role of the IFS measures must be specified.   
In some applications, there are important limit theorems, 
for example for financial derivative models both discrete and 
continuous pricing formulas are important. And it may be stated in 
the setting of generalized Bratteli diagrams, measurable setting.
We mean first of all 
financial derivative models, e.g., binomial models; vs 
continuous/measurable (cm), e.g., pricing of options via Ito 
calculus, see e.g., \cite{BrunickShreve2013}, 
\cite{KaratzasShreve1998}. In these applications, we also have
theorems to the effect that the continuous models are limits of 
discrete counterparts. Typically, the limit arguments involve the 
Central Limit Theorem from probability.
Another case of discrete models take the form of deep neural network
models, deep means a “large number of levels”, so many steps through
the levels in our diagrams.

The organization of the paper is the following. Section 
\ref{s:Basics} contains the basic definitions related to the 
concept of generalized Bratteli diagram (discrete case), 
and the description of various classes of measures on the path
space of a generalized Bratteli diagrams. We consider
tail invariant, shift-invariant, and Markov measures and their 
relations to a semibranching function system generated by a 
stationary generalized Bratteli diagram. 
In section \ref{s:discrete GBD}, we consider a semibranching 
function system $\{\tau_e\}$ defined on the path space of a 
stationary generalized Bratteli diagram and indexed by the
edge set $E$. We prove there one of the main results by giving
necessary and sufficient conditions on the existence of an
IFS measure for $\{\tau_e\}$, see Theorem \ref{thm IFS stat BD}.
Section \ref{sect MBD} focuses on 
measurable Bratteli diagrams. Since this notion is relatively new,
we give detailed definitions and discuss the properties of such
diagrams. Then we prove a measurable version of the main theorem 
about the existence of an IFS measure, Theorem \ref{thm IFS MBD}.

At the end of this introduction, we mention the literature 
that may be interesting for the reader. The literature on standard
Bratteli diagrams  their applications in 
dynamics is very extensive. We mention here the following 
pioneering papers 
\cite{HermanPutnamSkau1992}, \cite{GiordanoPutnamSkau1995}, 
\cite{DurandHostSkau1999}, a recent book \cite{Putnam2018}, and
surveys \cite{Durand2010}, \cite{BezuglyiKarpel2016},
\cite{BezuglyiKarpel2020}. 
Generalized Bratteli diagrams are less studied. The stationary
case uses the Perron-Frobenius theory for infinite matrices. 
We refer to the book \cite{Kitchens1998} and the literature there.
These diagrams are discussed in \cite{BezuglyiJorgensen2022} and
\cite{Bezuglyi_Jorgensen_Sanadhya_2022}. 
More references and numerous connections with other areas can be 
found therein. In particular, the following papers on IFS measures 
and fractals are related to the current paper
\cite{BratteliJorgensen1999}, \cite{BratteliJorgensen2002}, 
\cite{DutkayJorgensen2009}, \cite{DutkayJorgensen2010},
\cite{DutkayJorgensen2014a}, \cite{Jorgensen2006}, 
\cite{Jorgensen2018(book)}, \cite{RavierStrichartz_2016},
\cite{Alonso-Strichartz_2020}, \cite{CaoHassleretal_2021}.

\section{Basics on generalized Bratteli diagrams} \label{s:Basics}

We consider the fundamentals of path space for generalized Bratteli 
diagrams in this section. 
This notion was first introduced in 
\cite{BezuglyiDooleyKwiatkowski2006} under 
the name of Borel-Bratteli diagram. More detailed exposition of 
this concept can be found in \cite{BezuglyiJorgensen2022}. For the 
reader's convenience we give a concise version here. 

\subsection{Main definitions}\label{ss_main def}
In the introduction, we described the notion of a Bratteli diagram 
as an infinite graded graph. A natural extension
of this concept consists of consideration of diagrams with  
countably infinite levels.

\begin{definition}\label{def GBD} (\textit{Generalized Bratteli 
diagrams, vertices, edges, incidence matrices})
Let $V_0$ be a countable set (which can be identified with either 
$\N$ or $\Z$ if necessary).
Set $V_i = V_0$ for all $i \geq 1$, and $V = \bigsqcup_{i=0}^\infty 
V_i$.
A countable graded graph $B = (\mathcal{V, E})$ is called a 
\textit{generalized Bratteli diagram} 
if it has the following properties.

(i) The set of edges $\mathcal E$ of $B$ is represented as 
$\bigsqcup_{i=0}^\infty  
E_i$ where $E_i$ is the set of edges between the vertices of levels
$V_i$ 
and $V_{i+1}, i \geq 0$. 

(ii) The set $E(w, v)$ of edges $e$ between the vertices $w \in V_i$
and $v \in V_{i+1}$ is finite (or empty).
Let $f^{(i)}_{v,w} = |E(w, v)|$ where $|\cdot |$ denotes the 
cardinality of a set. It defines a sequence of infinite 
(countable-by-countable) 
\textit{incidence  matrices} $(F_i  : i \in \mathbb N_{0})$ with 
entries $F_i = (f^{(i)}_{v,w} : v \in V_{i+1}, w\in V_i),\ \   
f^{(i)}_{v,w}  \in \N_0.$

(iii) It is required that the matrices 
$F_i$ have at most \textit{finitely many non-zero entries} in each 
row. In general, we do not impose any restrictions on the columns 
of $F_i$.

(iv) The maps  $r,s : E \to V$ are defined on the diagram $B$: 
for every $e \in \E$, there are $w, v$ such that $e \in E(w, v)$; 
then $s(e) =w$ and
$r(e) = v$. They are called the \textit{range} ($r$) and 
\textit{source} 
($s$) maps. 

(v) For  every $w \in V_i, \; i \geq 0$,
there exist an edge $e \in E_i$ such that $s(e) = w$; for every 
$v\in V_i,\; i >1$, there exists an edge $e' \in 
E_{i-1}$ such that $r(e') = w$. In other 
words, every row and every column of the incidence matrix $F_i$ has 
non-zero entries.

(vi) If all entries of incidence matrices $F_n$ are zero or ones,
the  corresponding generalized Bratteli diagram is called a 0-1 
diagram.
\end{definition}

\begin{remark}
(1)  It follows from Definition \ref{def GBD} that the structure of 
every generalized 
Bratteli diagram is completely determined by a sequence of matrices 
$(F_n)$ such that every matrix $F_n$ satisfies (iii) and (iv). 
Indeed, the  entry $f_{v,w}^{(n)}$ indicates  the number of edges 
between the vertex $w \in V_n$ and vertex 
 $v\in V_{n+1}$. It defines the set $E(w, v)$; then one takes 
$$
E_n = \bigcup_{w\in V_n, v \in V_{n+1}} E(w, v).
$$
In this case, we write $B= B(F_n)$. If all $F_n = F$, the 
corresponding generalized Bratteli diagram $B(F)$ is called 
\textit{stationary.}

(2) If $V_0$ is a singleton, and each $V_n$ is a finite set, 
then we obtain the standard definition of a Bratteli diagram
originated in 
\cite{Bratteli1972}. Later it was used in the theory
of $C^*$-algebras  and dynamical systems for solving important 
classification  problems of Cantor dynamics and construction of
various models of homeomorphisms of a Cantor set (for references, 
see Introduction). 
\end{remark}

\begin{definition}\label{def path space} 
(\textit{Path space and cylinder sets})
A finite or infinite \textit{path} in a generalized
Bratteli diagram $B = (V,E)$ is a
sequence of edges $(e_i : i \geq 0)$ such that $r(e_i) = s(e_{i+1})$
for all $i\geq 0$. 
Denote by $X_B$ the set of all infinite paths. Every finite path 
$\ol e =
(e_0, \ldots , e_n)$ determines a \textit{cylinder subset} 
$[\ol e]$ of $X_B$:
$$
[\ol e] := \{x = (x_i) \in X_B : x_0 = e_0, \ldots, x_n = e_n\}.
$$
The collection of all cylinder subsets forms a base of neighborhoods
for a
topology on $X_B$. In this topology, $X_B$ is a  Polish 
zero-dimensional space, and every cylinder set is clopen. 
In general, $X_B$ is not locally compact.
But if the set $s^{-1}(v)$ is 
finite for every vertex $v \in \mathcal V$, then the path space
$X_B$ is locally compact.
\end{definition}

In the following remark, we formulate several
statements about properties of generalized Bratteli diagrams and
their path spaces. 

\begin{remark} \label{rem prop BD}
(1) If $x = (x_i)$ is a point in $X_B$, then it is obviously 
represented as the intersection of clopen sets:
\be \label{eq_ x as cylinders}
\{x\} = \bigcap_{n\geq 0} [\ol e]_n
\ee 
where $[\ol e]_n = [x_0, \ldots ,x_n]$. 
 
(2) Define a metric on $X_B$ compatible with the clopen topology:  
for $x = (x_i)$ and $y = (y_i)$ from $X_B$, we set
$$
\mathrm{dist}(x, y) = \frac{1}{2^N},\ \ \ N = \min\{i \in 
\N_0 : x_i \neq y_i\}.
$$

(3) We will \textit{assume} 
that the diagram $B$ is chosen so that the space $X_B$ 
\textit{has no isolated points}.  This means that every column of the
incidence matrix  $F_n, n \in \N_0,$ has more than one non-zero 
entry. 

(4) Let $\mathbbm 1 = (..., 1, 1,....)$ be the vector indexed by 
$v\in V_0$ such that every entry equals 1. Define $H^{(n)} := F_{n-1}
\cdots F_0 \mathbbm 1$. Let $E(V_0, v), v \in V_n,$ denote the set 
of all 
finite paths between $V_0$ and a fixed vertex $v\in V_n$. Then
$H^{(n)}_v = |E(V_0, v)|$.

(5) Let  $X_v^{(n)}$ be a subset of $X_B$ such that 
\be\label{eq_X_v}
X_v^{(n)} = \bigcup_{v_0 \in V_0}\bigcup_{\ol e \in E(v_0, v)} 
[\ol e].
\ee
For every level $V_n$, the collection $\{X_v^{(n)} : v \in V_n \}$
forms a  partition $\zeta_n$ of $X_B$ into disjoint clopen sets. 
Every set $X_v^{(n)}$ is a \textit{finite union} of cylinder sets. 
The 
number of the cylinder sets here   is exactly  $H^{(n)}_v$.
The sequence of partitions $(\zeta_n)$ is refining. 
According to \eqref{eq_X_v}, the cylinder sets from all $X_v^{(n)}$ 
generate the topology (and Borel $\sigma$-algebra) on  $X_B$. 
\end{remark} 

For a generalized Bratteli diagram $B$, define the 
\textit{tail equivalence relation} $\mathcal R$ on the path space 
$X_B$.

\begin{definition}\label{def tail} 
(\textit{Tail equivalence relation})
It is said that two infinite paths $x = (x_i)$ and $y = (y_i)$ are 
\textit{tail equivalent} if there
exists  $m \in \N$ such that $x_i = y_i$ for all $i \geq m$. Let 
$[x]_{\mathcal R} := \{ y \in X_B : (x,y) \in \mathcal R\}$ 
be the set of points 
tail  equivalent to $x$. We say that a point $x$ is 
\textit{periodic}
if $| [x]_{\mathcal R} | < \infty$.  If there is no periodic 
points, then the tail equivalence relation is called  
\textit{aperiodic}.
\end{definition}

Without loss of generality, we will consider generalized Bratteli 
diagrams with  aperiodic $\mathcal R$. 
Clearly, $\mathcal R$ is a hyperfinite countable  Borel equivalence
relation, see  \cite{DoughertyJacksonKechris1994} for definitions. 

\begin{definition}\label{def irreducible}
It is  said  that a generalized Bratteli diagram $B$ is 
\textit{irreducible} if, for any two vertices $v$ and 
$w$ and any $n \in \N_0$, there exists a level $V_m (m >n)$ 
such that $w \in V_n$ and $v\in V_m$
are connected by a finite path. This is equivalent to the property 
that, for any fixed $v,w$, there exists $m \in \N$ such that the 
product of matrices $F_{m-1} \cdots F_n$ has a non-zero 
$(v,w)$-entry.
\end{definition}

\subsection{Measures on the path space of a Bratteli diagram}
\label{ss_measures} 

In this subsection, we will consider two classes of Borel measures
on the path space $X_B$ of a generalized Bratteli diagram. They are
tail invariant measures and Markov measures. 

\begin{definition}\label{def tail inv m} 
(\textit{Tail equivalent measures})
Let $B = (\mathcal{\mathcal{V, E}})$ be a 
generalized Bratteli diagram, and $X_B$ the path space of  $B$. 
A Borel measure $\mu$ on $X_B$ (finite or $\sigma$-finite) is 
called \textit{tail invariant} if, for any two finite paths  $\ol e$
and $\ol e'$ such that $r(\ol e) = r(\ol e')$, one has 
\be\label{eq_inv measures}
\mu([\ol e]) = \mu([\ol e']),
\ee
where $[e]$ and $[e']$ denote the corresponding cylinder sets. 
\end{definition}

If $\mu(X_B) = 1$, then the property of tail invariance means that 
the probability to arrive
at a vertex $v \in V_n$ does not depend on a starting point $w \in 
V_0$ and does not depend on the path connecting $w$ and $v$.

Let $\mu$ be a Borel tail invariant measure on $X_B$. Relation 
\eqref{eq_inv measures}  defines a  sequence of
non-negative  vectors $(\mu^{(n)})$ where $\mu^{(n)} =
( \mu^{(n)}_v : v \in V_n)$:
\be\label{eq def og mu(n)}
\mu^{(n)}_v = \mu([\ol e]), \ \ \ \ol e\in E(V_0, v), \ v \in V_n.
\ee
Because $\mu$ is tail invariant the value $\mu^{(n)}_v $ does
not depend on the choice of $\ol e\in E(V_0, v)$. 

The following theorem is a key tool in the study of tail invariant 
measures, see \cite{BezuglyiKwiatkowskiMedynetsSolomyak2010}, 
\cite{Durand2010}, \cite{BezuglyiKarpel2016}, 
\cite{BezuglyiJorgensen2022}.
\begin{theorem}\label{thm inv measures} 
Let $B =(\mathcal V,\E)$ be  a generalized Bratteli diagram defined
by a sequence $(F_n)$ of incidence matrices.
Let $\mu$ be a Borel probability measure  on the path space $X_B$ 
of 
$B$ which is tail invariant. Then the corresponding sequence of 
vectors 
$\mu^{(n)}$ (defined as in \eqref{eq def og mu(n)}) satisfies the 
property 
\be\label{eq_inv meas via A_n}
A_n \mu^{(n+1)} = \mu^{(n)}, 
\ee
where $A_n = F_n^T$ is the transpose of $F_n$.

Conversely, if a sequence of vectors $\mu^{(n)}$ satisfies 
\eqref{eq_inv meas via A_n}, then it defines  a unique  tail 
invariant measure $\mu$.
 
 The theorem remains true for $\sigma$-finite measures $\nu$ 
 satisfying the property 
  $\nu([\ol e]) < \infty$ for every cylinder set $[\ol e]$.
\end{theorem}
 
The other interesting class of measures on Bratteli diagrams is 
Markov
measures. In the context of Bratteli diagrams, these measures were 
considered in \cite{DooleyHamachi2003}, \cite{Renault2018}, 
\cite{BezuglyiJorgensen2022} and some other papers.

\begin{definition}\label{def Mark meas} (\textit{Markov measures})
Let $B = (\mathcal{V, E})$ be  a generalized Bratteli diagram
constructed by 
a sequence of incidence matrices $(F_n)$.  Let $q = (q_{v})$ be a 
strictly positive  vector, $q_v >0, v\in V_0$, and let $(P_n)$ 
be a sequence of non-negative infinite matrices with entries 
$(p^{(n)}_{v,e})$ where $v \in V_n, e \in E_{n}, n= 0, 1, 2, 
\ldots $. To define a
\textit{Markov measure} $m$, we require that the sequence $(P_n)$
 satisfies the  following properties:
\begin{equation}\label{defn of P_n}
(a)\ \ p^{(n)}_{v,e} > 0 \ \Longleftrightarrow \ (s(e) = v); \ \ \ \
(b)\ \  \sum_{e : s(e) = v} p^{(n)}_{v,e} =1.
\end{equation}
Condition \eqref{defn of P_n}(a) shows that $p^{(n)}_{v,e}$ is 
positive 
only on the edges outgoing from the vertex $v$, and therefore the
 matrices $P_n$ and $A_n =F_n^T$ share the same set of zero entries.
 For any cylinder set $[\overline e] = [(e_0, e_1, \ldots , e_n)]$
generated by the path $\ol e$ with $v =s(e_0) \in V_0$, we set
\begin{equation}\label{eq m([e])}
m([\overline e]) = q_{s(e_0)}p^{(0)}_{s(e_0), e_0} \cdots 
p^{(n)}_{s(e_n), e_n}.
\end{equation}
Relation \eqref{eq m([e])} defines the value of the measure $m$ 
of the set $[\ol e]$.
By \eqref{defn of P_n}(b),  this measure satisfies the 
\textit{Kolmogorov consistency condition} and can be extended to the
$\sigma$-algebra of  Borel sets. 
To emphasize that $m$ is generated by a sequence of 
\textit{stochastic matrices}, we will also write $m = m(P_n)$.

If all stochastic matrices $P_n$ are equal to a matrix $P$, then 
the corresponding measure $m(P)$ is \textit{called stationary Markov
measure.}
\end{definition}

We refer to \cite{BezuglyiJorgensen2022} for a detailed study of 
Markov measures. We mention here only the following result.

\begin{theorem}\label{thm inv meas is Markov}
Let $\nu$ be a tail invariant probability 
measure on the path space $X_B$  of a generalized Bratteli diagram
$B = (\mathcal{V, E})$. Then there
exists a sequence of Markov matrices $(P_n)$ such that $\nu = 
m(P_n)$.
\end{theorem}

For stationary generalized Bratteli diagrams, we can find explicit
formulas for tail invariant measures. In the following statement,
we use the terminology from the Perron-Frobenius theory, see
\cite{Kitchens1998} for details. 

\begin{theorem} \label{thm inv meas stat BD} 
\cite{BezuglyiJorgensen2022}
(1) Let $B = B(F)$ be a stationary Bratteli diagram such 
that the incidence matrix $F$ (and therefore $A = F^T$) is 
irreducible, aperiodic, and recurrent. Let $t = (t_v : v \in V_0)$ 
be a right eigenvector  corresponding to the Perron eigenvalue 
$\lambda$ for $A$, $At = \lambda t$. 
Then there exists a tail 
invariant measure $\mu$ on the path space $X_B$ whose values on 
cylinder sets are determined by the following rule: for every finite
path $\ol e(w, v)$ that begins at $w \in V_0$ and 
 terminates at $v \in V_n$, $n \in \N_0$, we  set
 \be\label{eq inv meas stat BD}
\mu^{(n)}_v =  \mu([\ol e(w, v)]) = \frac{t_v}{\lambda^{n}}. 
 \ee

(2) The measure $\mu$ is finite if and only if the right eigenvector
$t = (t_v)$ has the property $\sum_v t_v <\infty$.
\end{theorem}

In particular, $\mu(X_w^{(0)}) = t_w, w \in V_0$, and
\begin{equation}
    \mu(X_v^{(n)}) = H_v^{(n)} \frac{t_v}{\lambda^{n}}.
\end{equation}

\subsection{Semibranching function systems on Bratteli diagrams}
\label{subs s.b.s.}

Here we give the definition of a semibranching function system 
following 
\cite{MarcolliPaolucci2011} and \cite{BezuglyiJorgensen2015}.
This notion was used in the literature, in particular, for the 
construction of representations
of Cuntz-Krieger algebras, see \cite{BezuglyiJorgensen2015},  
\cite{FarsiJKP2018a}, \cite{FarsiJKP2018b}.

\begin{definition}\label{s.f.s.} 
(\textit{Semibranching function systems and coding maps})

(1) Let $(X,\mu)$ be a probability 
measure space with non-atomic measure $\mu$. We consider a  family 
$\{\sigma_i : i\in \Lambda\}$ of one-to-one $\mu$-measurable maps 
 indexed by a finite (or countable) set $\Lambda$. 
The family  $\{\sigma_i\}$ is called a \textit{semibranching 
function
system  (s.f.s.)} if the following conditions hold:

(i) $\sigma_i$ is defined on a subset $D_i$ of $X$ and takes values 
in $R_i = \sigma_i(D_i)$ such that $\mu(R_i \cap R_j) = 0$ for 
$i\neq j$ and $\mu(X\setminus \bigcup_{i\in \Lambda} R_i) = 0$;

(ii) the measure $\mu \circ \sigma_i$ is equivalent to  $\mu$ and,
i.e., 
$$
\rho{_\mu}(x, \sigma_i) := \frac{d\mu\circ \sigma_i}{d\mu}(x) > 0 
\ \ \mbox{for $\mu$-a.e. $x\in D_i$};
$$

(iii) there exists an endomorphism $\sigma : X \to X$ (called a 
\textit{coding map}) such that $\sigma\circ \sigma_i(x) = x$ for 
$\mu$-a.e. $x\in D_i, \ i \in \Lambda$.

If, additionally to properties (i) - (iii), we have $\bigcup_{i\in 
\Lambda}D_i = X$ ($\mu$-a.e.),  then  the s.f.s. $\{\sigma_i : i\in 
\Lambda\}$ is called \textit{saturated.}

(2) It is said that a saturated s.f.s. satisfies 
\textit{condition C-K}
\footnote{C-K stands for Cuntz-Krieger.}
if for any $i\in \Lambda$ there exists a subset $\Lambda_i \subset 
\Lambda$ such that up to a set of measure zero
$$
D_i = \bigcup_{j\in \Lambda_i} R_j.
$$
In this case, condition C-K defines a 0-1 matrix $\wt A$ by the 
rule:
\begin{equation}\label{C-K defines A}
\wt a_{i,j} =1 \ \ \Longleftrightarrow \ \ j\in \Lambda_i, \ \ i \in
\Lambda.
\end{equation}
Then the matrix $\wt A$ is of the size $|\Lambda|\times |\Lambda|$.
\end{definition}

\textbf{Semibranching function system associated with a generalized
stationary Bratteli diagram}.

Let $B$ be a generalized stationary 0-1 Bratteli diagram. We
construct an s.f.s. 
$\Sigma$ which is defined on the path space $X_B$ endowed with a 
Markov measure $m$. As we will see below, this measure must 
have some
additional properties to satisfy Definition \ref{s.f.s.}. 
The role of the index set  $\Lambda$ for this s.f.s. 
is played by the edge set $E$ which is the set of edges between 
any two consecutive levels of $B$.  For 
any $e \in E$, we denote
\begin{equation}\label{defn of D_e}
D_e = \{y = (y_i) \in X_B : s(y) = s(y_0) = r(e)\},
\end{equation}
\begin{equation}\label{defn of R_e}
 R_e = \{y = (y_i) \in X_B : y_0 = e\}.
\end{equation}
We see that $D_e$ depends on $r(e)$ only so that $D_e = D_{e'}$
if and only if $r(e) = r(e')$.

The collection of maps $\{\tau_e : e \in E\}, \tau_e : D_e \to 
R_e$, is defined by the formula
\begin{equation}\label{defn of sigma_e}
     \tau_e (y) :=  (e, y_0, y_1, .... ), \quad y = (y_i).
\end{equation}
Since $s(y_0) = r(e)$, the map $\tau_e$ is well defined on $D_e$.

\begin{remark}\label{rem contractive}
We defined the metric $\mbox{dist}$ in Remark \ref{rem prop BD}. 
It follows from the definition of $\tau_e$ that 
$$
\mbox{dist}(\tau_e(x), \tau_e(y) ) = \frac{1}{2}  \mbox{dist}
(x, y), \ e \in E,
$$
that is $\tau_e$ is a contractive map for every $e$.
\end{remark}

\begin{proposition}
The system $\{D_e, R_e, \tau_e : e \in E\}$ is a saturated 
s.f.s. on the 
path space $X_B$ of a generalized stationary Bratteli diagram 
satisfying conditions (i), (iii) of Definition 
\ref{s.f.s.} and  the C-K condition. 
\end{proposition}

\begin{proof}
Let $\sigma : X_B \to X_B$ be defined as follows: 
for any $x = (x_i)_{i \geq 0} \in X_B$,
\begin{equation}\label{defn of sigma}
     \sigma(x) :=  (x_1, x_2, ...)
\end{equation}
It follows from (\ref{defn of sigma_e}) and (\ref{defn of sigma}) 
that the map $\sigma$ is onto and 
$$
\sigma\circ\tau_e(x) = x,\ \ \  x \in D_e;
$$
Hence $\sigma$ is a coding map.

We deduce from (\ref{defn of R_e}) that $\{R_e : 
e\in E\}$ constitutes a partition of $X_B$ into clopen sets. 
Relation
(\ref{defn of D_e}) implies that  $\{\tau_e : e \in E\}$ 
is a saturated s.f.s. Moreover, we claim that it satisfies 
condition C-K. Indeed,
\begin{equation}\label{checking C-K}
     D_e = \bigcup_{f : s(f) = r(e)} R_f,\ \ \ e \in E,
\end{equation}
because $y = (y_i) \in D_e \ \Longleftrightarrow \ s(y_0) = r(e) \ 
\Longleftrightarrow \ \exists f = y_0\ \mbox{such\ that}\ y = 
(f, y_1, ...) \ \Longleftrightarrow \ y \in 
\bigcup_{f : s(f) = r(e)} R_f$. Thus, $\Lambda_e = \{f : s(f) =
r(e)\}$, see \eqref{C-K defines A}.

Relation (\ref{checking C-K}) shows that the non-zero entries of 
the  0-1 matrix $\wt A$ from Definition \ref{s.f.s.} are defined by 
the rule:
\be\label{eq-matrix tilde A}
(\wt a_{e,f} =1) \ \Longleftrightarrow \ (s(f) = r(e)). 
\ee

We observe that the matrix $\wt A$ has the following property:
there are finitely many nonzero entries in every column of $\wt A$,
but the rows of $\wt A$ may contain infinitely many nonzero entries. 


Next, we observe that $\sigma: X_B \to X_B$ is a finite-to-one 
continuous map. Indeed,
$$
|\sigma^{-1}(x)| = |r^{-1}(r(x_1))| = \sum_{u\in V_0} f_{v,u}.
$$
The latter is finite by Definition \ref{def GBD}. 
\end{proof}

Thus, it remains to find out under what conditions property (ii) of 
Definition \ref{s.f.s.} holds. We consider here two classes: tail 
invariant measures and Markov measures.

Let $m$ be a Borel probability measure
on $X_B$. Since $X_B$ is naturally partitioned 
into a refining sequence of clopen partitions $\mathcal Q_n$
formed by cylinder sets of length $n$, we can apply de Possel's 
theorem (see, for instance, \cite{ShilovGurevich1977}). 
We have that for $m$-a.a $x$,
\begin{equation}\label{eq:Possel's}
\rho_m(x, \tau_e) = \lim_{n\to \infty} \frac{m(\tau_e([\overline
e(n)])}{m[\overline e(n)]}
\end{equation}
where $\{x\} = \bigcap_n [\overline e(n)]$.

\textit{Tail invariant measure}. 
We first consider  the case when $m$ is the tail invariant measure 
$\mu$ determined in Theorem \ref{thm inv meas stat BD}. Let 
$At = \lambda t$ where $t = (t_v)$.
If $\overline f = (f_0, f_1, ... , f_n) \in D_e$, then $r(\ol f) 
= r(f_n)$ and 
$\tau_e(\overline f) = (e, f_0, ... , f_n)$. By 
(\ref{eq inv meas stat BD}), we have  
$$
\mu([\overline f ]) = \frac{t_{r(\overline f)}}{\lambda^{n}},\ \ \ 
\mu(\tau_e([\overline f ])) = 
\frac{t_{r(\overline f)}}{\lambda^{n+1}},
$$
and therefore
\begin{equation}\label{RN for mu}
     \rho_{\mu}(x, \tau_e) = \lambda^{-1}.
\end{equation}

\textit{Stationary Markov measure}. Let $m$ be a stationary Markov
measure determined by a stochastic matrix $P$, $m = m(P)$ as in
\eqref{eq m([e])}. If $\{x\} =\bigcap_n [f(n)] \in D_e$, then
$$
m([\overline f(n)]) = q_{s(f_0)}p_{s(f_0), f_0} \cdots 
p_{s(f_n), f_n},
$$
$$
m(\tau_e([\overline f(n)])) = q_{s(e)} p_{s(e), e}p_{s(f_0), f_0} 
\cdots p_{s(f_n), f_n},
$$
and the Radon-Nikodym derivative can be found by
\begin{equation}\label{RN for nu}
 \rho_{m}(x, \tau_e) = \frac{q_{s(e)} p_{s(e),e}} {q_{s(f_0)}}.
\end{equation}

It follows from (\ref{RN for mu}) and (\ref{RN for nu}) that  
$\rho_{m}$ is positive on $D_e$ if and only if all entries of the 
vector $q = (q_v)$ are positive. The latter means that the support 
of $m$ is the entire space $X_B$.

\textit{Non-stationary Markov measure}. In this case, we need 
some additional conditions to guarantee that the Radon-Nikodym 
derivative  $\rho_m(x, \tau_e)$ is positive on $D_e$. As above, we
represent $x = (x_i)\in D_e$ by means of the sequence 
$[\overline f(n)] = [(f_0, f_1, ... , f_n)]$ such that $x_i = f_i, 
i= 0,1, ... ,n$, and find 
$$
m([f(n)]) = q_{s(f_0)} p^{(0)}_{s(f_0), f_0} \cdots 
p^{(n)}_{s(f_n), f_n}
$$
and
$$
m(\tau_e([f(n)])) = q_{s(e)} p^{(0)}_{s(e), e}p^{(1)}_{s(f_0),
f_0} \cdots p^{(n+1)}_{s(f_n), f_n}.
$$

A direct computation gives the following result. 

\begin{lemma}\label{RN for m}
Let $m$ be a Markov measure on the path space of a 
generalized stationary 0-1 
Bratteli diagram. Then $\rho_m(x, \tau_e) >0$ on $D_e$ 
if and only if
\begin{equation}\label{product convergence}
   0<  \prod_{i=1}^\infty \frac{p^{(i+1)}_{s(f_i), 
   f_i}}{p^{(i)}_{s(f_i), f_i}}  <\infty
\end{equation}
for any $x = \bigcap_n[\overline f(n)] \in D_e$
\end{lemma}

A Markov measure satisfying (\ref{product convergence}) is called a 
\textit{quasi-stationary measure}.  Condition 
(\ref{product convergence}) appeared first in 
\cite{DutkayJorgensen2014}, \cite{DutkayJorgensen2015} in a 
different context.

We summarize the above results  in the following theorem.

\begin{theorem}\label{example summary}
Given a  generalized stationary 0-1 Bratteli diagram $B$ with the 
edge set $E$, the collection of maps $\{\tau_e : D_e \to R_e\},\ 
e \in E$, defined in \eqref{defn of sigma_e} on the space  $(X_B, 
m)$, forms  a saturated s.f.s. $\Sigma$ satisfying C-K condition 
where the Markov measure $m$ is either the tail invariant measure 
$\mu$, or a stationary Markov measure $m(P)$ of 
full support, or a quasi-stationary measure Markov measure of full
support.
\end{theorem}

\begin{remark} 
Let $B$ be a generalized stationary Bratteli diagram of 
\textit{bounded size}, see \cite{BJKS_2022} for the definition. 
In particular, if the incidence matrix $A$ is banded, then $B$ is 
of bounded size. In this case, the definition of the 
Cuntz-Krieger algebra 
$\mathcal O_A$ can be given similar to the case of finite matrices.
Then we can use the methods developed in 
\cite[Theorem 4.12]{BezuglyiJorgensen2015}
to construct a representation of the Cuntz-Krieger algebra 
$\mathcal O_A$ generated by the s.f.s. defined in 
Theorem \ref{example summary}. We omit the details. 
\end{remark}

\subsection{Shift invariant measures on stationary Bratteli 
diagrams}\label{ss_shift inv}

Let $B =(\mathcal{V,E})$ be a generalized stationary Bratteli
diagram defined 
by the incidence infinite matrix $F$ and $A = F^T$. 
In this subsection, we discuss $\sigma$-invariant  measures on 
$X_B$. We will consider two cases, tail invariant measures and
Markov measures.

Recall that, for every $v \in V_1$, the row sum $H^{(1)}_v = 
\sum_{w \in V_0} f_{v,w}$ is 
finite, see Remark \ref{rem prop BD} for notation. 
In \eqref{defn of sigma}, we defined a finite-to-one endomorphism
$\sigma$ acting on the path space $X_B$. For $x = (e_0. e_1, ...)$,
we have 
\be \label{eq preimage for sigma}
\sigma^{-1}(x) = \{y = (y_i) \in X_B : r(y_0) = s(e_0), y_i = 
e_{i-1}, i \geq 1\}
\ee 
and $|\sigma^{-1}(x)| = H^{(1)}_{s(e_0)}$.

\textit{Tail invariant measures}. There is a special case of 
a generalized stationary Bratteli diagram such that the 
tail invariant measure is also shift-invariant.

\begin{proposition}
Let the matrix $A$ has a Perron-Frobenius eigenpair $(\xi,\lambda)$.
Let $\mu$ be the tail invariant measure on $X_B$ defined as in
Theorem \ref{thm inv meas stat BD}. Then $\mu$ is $\sigma$-invariant
if and only if $H^{(1)}_v = \lambda$ for all $v\in V_0$.
\end{proposition}

\begin{proof} For a cylinder set $[\ol e] = [e_0, e_1, ..., 
e_n]$, calculate $m([\ol e])$ and $m(\sigma^{-1}([\ol e]))$.
it follows from  \eqref{eq preimage for sigma} that
\begin{equation}
    \sigma^{-1}([\ol e]) = \bigcup_{f : r(f) =s(e_0)} 
    [f, e_0, ..., e_n] = \bigcup_{f : r(f) =s(e_0)} [f, \ol e]
\end{equation}
and this union is disjoint. Let $\lambda$ and $\xi$ be a 
Peron-Frobenius eigenpair, $A\xi = \lambda\xi$. 
By Theorem \ref{thm inv meas stat BD}, we have 
$$
\mu([\ol e]) = \frac{\xi_v}{\lambda^n}, \qquad 
\mu([f, \ol e]) = \frac{\xi_v}{\lambda^{n+1}},\ v = r(\ol e)
$$
Then 
$$
\mu(\sigma^{-1}[\ol e]) = \sum_{f : r(f) = s(e_0)} 
\mu ([f, \ol e]) =  \sum_{f : r(f) = s(e_0)} 
\frac{\xi_v}{\lambda^{n+1}} = |\sigma^{-1} ([\ol e])| 
\frac{\xi_v}{\lambda^{n+1}} = H^{(1)}_{w} 
\frac{\xi_v}{\lambda^{n+1}}
$$
where $w = s(e_0) \in V_1$. By de Possel's theorem 
\eqref{eq:Possel's}, 
$$
\frac{d\mu\circ\sigma^{-1}}{d\mu}(x) = \lim_{n \to\infty}
\frac{\mu\circ\sigma^{-1}(\ol e]_n)}{\mu([\ol e]_n)} =
H^{(1)}_{w} \lambda^{-1}
$$
where $x = \bigcap_n [\ol e]_n$. Therefore, 
$$
\mu\circ\sigma^{-1} = \mu \ \Longleftrightarrow \ 
H^{(1)}_{w} = \sum_{u \in V_0} f_{w, u} = \lambda, \ \forall 
w \in V_0.
$$
\end{proof}

\begin{remark}
If $\max\{ H^{(1)}_v : v \in V_1\} = M <\infty$, then 
$$
\lambda^{-1} \leq \frac{d\mu\circ\sigma^{-1}}{d\mu}(x) \leq 
M\lambda^{-1}.
$$
This means that $\mu$ is equivalent to a $\sigma$-invariant measure
$\mu'$. In particular, this is true for bounded size diagrams. 
\end{remark}
\vskip 2mm

\textit{Markov measures}. Consider first a stationary Markov measure
on the path space of a generalized stationary Bratteli diagram,
see Definition \ref{def Mark meas}. For simplicity we will work
with a $0-1$ diagram; the general case is considered similarly.

Let $q = (q_v : v \in V_0)$ be a positive probability vector 
(called initial distribution). Let $P$ be a Markov matrix. For
a cylinder set $[\ol e] = [e_0, e_1,..., e_n]$, we use 
\eqref{eq m([e])} to determine the value of the corresponding 
Markov measure $m(P)$.

\begin{theorem}
For $B, q, P, m(P)$ as above, the measure $m = m(P)$ is 
$\sigma$-invariant if and only if $q P = P$. 
\end{theorem}

\begin{proof} Suppose that $q P = q$. 
We know that $m([\ol e]) = q_{s(e_0)}p_{s(e_0), r(e_0)} \cdots 
p_{s(e_n), r(e_n)}$ and 
\be \label{eq: qP = q}
\ba
m(\sigma^{-1}[\ol e]) =  & \sum_{f : r(f) = s(e_0)} 
m([f, \ol e]) \\
= & \sum_{f : r(f) = s(e_0)} q_{s(f)} p_{s(f), r(f)}
p_{s(e_0), r(e_0)} \cdots p_{s(e_n), r(e_n)} \\
=& q_{r(f)} p_{s(e_0), r(e_0)} \cdots p_{s(e_n), r(e_n)} \\
=& m([\ol e])
\ea
\ee
because $q_{r(f)} = q_{s(e_0)}$ and $q$ is $P$-invariant. Since
$\ol e$ is arbitrary, we see that $m$ is $\sigma$-invariant.

The converse statement follows from \eqref{eq: qP = q}: the equality
$m([\ol e]) = m(\sigma^{-1}[\ol e])$ implies $q P =q$.
\end{proof}

It remains to consider a non-stationary Markov measure $m$ defined
on the path space of a generalized stationary Bratteli diagram $B$.
In this case, the measure $m$ is defined by a sequence of Markov 
matrices $(P_n)$ and an initial distribution $q = (q_v)$:
$$
m([e_0, ..., e_n]) = q_{s(e_0)}p^{(0)}_{s(e_0), r(e_0)} \cdots 
p^{(n)}_{s(e_n), r(e_n)}.
$$
\begin{proposition}
Suppose that $q P_0 =q$. Then the measures $m = m(P_n)$ is 
$\sigma$-invariant if and only if 
$$
\prod_{i=1}^\infty \frac{p^{(i+1)}_{s(e_i), r(e_i)}} 
{p^{(i)}_{s(i_i), r(e_i)}} = 1. 
$$
\end{proposition}

\begin{proof}
Using the same technique, we can show that, under the condition 
$q P = q$, 
$$
\frac{d\mu\circ\sigma^{-1}}{d\mu}(x) = \prod_{i=1}^\infty 
\frac{p^{(i+1)}_{s(e_i), r(e_i)}} {p^{(i)}_{s(i_i), r(e_i)}}
$$
where $x = \bigcap_n [e_0, ..., e_n$. We omit the details.
\end{proof}

\section{IFS measures on discrete generalized Bratteli diagrams} 
\label{s:discrete GBD}

In this section, we consider the notion of \textit{iterated 
function system (IFS)}. This concept have been discussed in many 
books and papers. We mention here only several related references 
such as \cite{Hutchinson1995},
\cite{HutchinsonRueschendorf1998}, \cite{Jorgensen2006}, 
\cite{Barnsley2006}, \cite{Jorgensen2018(book)}, 
\cite{DutkayJorgensen2007}, \cite{DutkayJorgensen2009}, 
\cite{JKS_2011}, \cite{MorrisSert2021}. 
Our goal is to describe IFS measures defined on the path space of
a generalized Bratteli diagram.

\subsection{Iterated function systems and measures} 

By an \textit{endomorphism} $\sigma$ we mean a finite-to-one 
(or countable-to-one) Borel map of a standard Borel space $(X, \B)$
onto itself. In this case, there exists a 
family of one-to-one maps $\{\tau_i\}_{i\in \Lambda}$ such that 
$\tau_i : X \to X$ and $\sigma \circ \tau_i = \mbox{id}_{X}$ where 
$\Lambda$ is at most countable. The maps 
$\tau_i$ are called the \emph{inverse branches} for $\sigma$. The 
collection of maps $(\tau_i : 1 \in \Lambda)$ gives an 
example of \textit{iterated function system (IFS)}. 
In general, an IFS 
is defined by a collection of  maps $\{\tau_i : i\in \Lambda\}$ 
of a complete metric space (or a compact space) such that 
$\tau_i$'s are continuous (or even contractions). In this work, we
focus on the following features of the family $\{\tau_i : i \in
\Lambda\}$: (i) each $\tau_i$ is a one-to-one map defined on a 
subset $D_i$ of $X$, and (ii) there exists an endomorphism $\sigma$
of $X$ such that $\sigma\circ\tau_i = \mathrm{id}_X$ for each $i$.

The general theory of \textit{infinite iterated function systems} is
more detailed and  requires additional assumptions 
(see, for example, the expository article \cite{Mauldin1998}). 
As an example, one can consider the Gauss map 
$x \mapsto \{1/x\}$ of the unit interval
with the piecewise monotone inverse branches $\tau_i : 1/(i+x)$
defined on $(1/(i+1), 1/i)$. 

\begin{definition}\label{def_IFS measure} (\textit{IFS measures})
Suppose that 
$(\tau_i : i \in \Lambda)$ is a given IFS on a Borel space 
$(X, \B)$. Let $p = (p_i : i \in \Lambda)$ be a strictly positive  
vector indexed by a countable set $\Lambda$.  
A measure $\mu_p$ on  $(X, \B)$ is called an \emph{IFS measure} 
for $(\tau_i)$ if
\be\label{eq IFS measure def}
\mu_p =  \sum_{i\in \Lambda} p_i\; \mu_p\circ \tau_i^{-1},
\ee
or, equivalently,
$$
\int_X f(x)\; d\mu_p(x)  =  \sum_{i\in \Lambda} p_i
\int_X f(\tau_i(x))\; d\mu_p(x), \  \qquad f\in L^1(\mu_p).
$$
\end{definition}

The main tool in the study of an IFS $(\tau_i : i \in \Lambda)$ on a
space $X$ is a realization of the IFS as the full one-sided shift 
$S$ on a symbolic product space $\Omega$. Then any $S$-invariant and
ergodic product-measure $\mathbb P$ on $\Omega$ can be pulled back 
to $X$. This construction gives ergodic invariant measures for IFSs.

In the next subsection, we will find an explicit method for 
construction of an IFS measure on the path space of a generalized
stationary Bratteli diagram. 

We discuss here a method that leads to construction of IFS measures.
This approach works perfectly for many specific applications under
some additional conditions on $X$ and maps $\tau_i$. 

Suppose that $(X; \tau_i, i \in \lambda)$ is a given IFS.
Let $\Omega$ be the infinite direct product 
$$
\Omega = \prod_{i\in \N_0} \Lambda_i,\qquad \Lambda_i = \Lambda.
$$
For  $\omega = (\omega_1, \omega_2, ...) \in 
\Omega$, let $\omega|_n $ denote the finite word $(\omega_0, ... , 
\omega_n)$. Then,  $\omega|_n$  defines  a map $\tau_{\omega|_n}$
acting on $X$ by the formula:
$$
\tau_{\omega|_n} (x) := \tau_{\omega_0} \cdots \tau_{\omega_n}
(x), \qquad x \in X, \ \  \omega \in \Omega, \ n \in \N_0.
$$
It is said that $\Omega$ is an \emph{encoding space} if, for every 
$\omega\in \Omega$,
\be\label{eq singleton}
F(\omega) = \bigcap_{n \geq 1}\tau_{\omega|_n}(X)
\ee
is a singleton. Relation \eqref{eq singleton} defines a Borel map 
$F : \Omega \to F(\Omega)$.  
If each $\tau_i$ is a contraction and $X$ is a complete metric, 
then a coding map $F :\Omega\to X$ always exists.

Let $S$ be the left shift on $\Omega$:
$$
S(\omega_0, \omega_1, ... ) = (\omega_1, \omega_2, ... ).
$$
The inverse branches $s_i, i \in \Lambda,$ of $S$ are 
$$
s_i (\omega_0, \omega_1, ...) = (i, \omega_0, \omega_1, ...).
$$
Clearly,
$$
s_i (\Omega) = C(i) = \{\omega \in \Omega : \omega_0 = i\},
$$
and the space $\Omega $ is partitioned into the sets $C(i), \in 
\Lambda$.

Let $p = (p_i : i \in \Lambda)$ be a positive probability vector. 
It defines the product measure $\mathbb P = p \times p \times \cdots$
on $\Omega$. 
We observe that the maps $(s_i : i\in \Lambda)$ constitute 
an IFS on $\Omega$ such that $\mathbb P$ is an IFS measure:
\be\label{eq P is IFS}
\mathbb P = \sum_{i\in \Lambda} p_i \; \mathbb P\circ s_i^{-1}.
\ee

The following result shows how IFS measures arise.

\begin{proposition} Suppose that $(X; \tau_i, i\in \Lambda)$ 
is an IFS  that admits a coding map $F : \Omega \to X$. Let 
$p = (p_i)$ be  a probability vector generating the product measure 
$\mathbb P = p \times p \times \cdots$. Then the measure
$\mu := \mathbb P\circ F^{-1}$ is an IFS measure satisfying
$$
\mu = \sum_{i=1}^N p_i \mu\circ \tau_i^{-1}.
$$
Moreover, if $F$ is continuous, then  $\mu$ has full support.
\end{proposition}

The proof can be found in \cite{BezuglyiJorgensen2018(book)}.

In Section \ref{sect MBD}, we will consider a analogue of the above 
construction for measurable Bratteli diagrams.

\subsection{IFS measures on generalized stationary Bratteli 
diagrams}

In this subsection we discuss an explicit formula for an IFS measure
on the path space of a generalized stationary 0-1 Bratteli diagram 
$B = (\mathcal V,\mathcal E)$.

In the following remark, we show that the requirement to be 
a 0-1 Bratteli diagram is not restrictive. 

\begin{remark}
Every generalized stationary Bratteli diagram $B= (\mathcal V, \E)$
can be represented as a 0-1 
Bratteli diagram. Indeed, we can use the matrix $\wt A$
defined in \eqref{eq-matrix tilde A} to built a new generalized 
Bratteli diagram $\wt B = (\wt V, \wt E)$. 
Note that $\wt A$ is a 0-1 matrix such
that every row of $\wt F = \wt A^T$ has finitely many non-zero 
entries so that $\wt F$ can be viewed as an incidence matrix 
of $\wt B$. In this case, the set of vertices $\wt V$ for each 
level is coincides with $E$ and two vertices $\wt v =f \in E$ and  
$\wt w = f \in E$ are connected by a singe edge if and only if 
$r(e) = s(f)$. Moreover, the path spaces $X_{\wt B}$ coincides 
with $X_B$.

The above argument can be applied to any non-stationary 
generalized Bratteli diagram. 
\end{remark}

We can now apply the results of Subsection \ref{subs s.b.s.} 
and construct an s.f.s. $\Psi = (X_B; \tau_e, e \in E)$ where 
$\tau_e : D_e \to R_e$ is defined in \eqref{defn of sigma_e}.

\begin{theorem}\label{thm IFS stat BD}
Let $B = (\mathcal V,\mathcal E)$ be a  generalized stationary 
0-1 Bratteli diagram, $p = (p_e : e \in E)$ a probability vector,  
and $\Psi$ an s.f.s. defined by $\tau_e : D_e \to R_e$, $e \in E$. 
We identify every edge 
 $e \in E$ with the pair of vertices $(s(e), r(e))$. 
Let $\nu$ be a Borel probability full measure on $X_B$; we define a
probability positive vector $q = (q_v : v \in V_0)$ by setting
$q_v = \nu([w])$ where $[w] = \{ x in X_B : s(x) = w\}$.
Then 
\begin{equation}\label{eq_meas nu}
    \nu = \sum_{e\in E} p_e \; \nu\circ \tau_e^{-1}
\end{equation}
if and only if 
\begin{equation}\label{eq: q =qP}
    q_w = \sum_{w \in V_0}  p_{w,v} q_v = \sum_{e : s(e) = w} 
    p_{s(e), r(e)} q_{r(e)}, 
\end{equation}
that is $Pq = q$ where the matrix $P$ has the entries  
$(p_{s(e), r(e)} : e \in E) $.
\end{theorem}

\begin{proof} We use in the proof notation from \eqref{defn of D_e}
- \eqref{checking C-K}. 
We note that relation \eqref{eq_meas nu} will be proved if we
show that it holds for any cylinder set $C \subset X_B$. Recall 
that 
it follows from the definition of $\tau_e$ that $\tau_e^{-1}$ is
uniquely determined on $R_e$ and $\tau_e^{-1}(e, y_1, y_2, ...) =
( y_1, y_2, ...)$. 

Let $q = (q_w) >0$ be a solution to $q = P q$.  Define
$\nu([w]) = q_w$, $w\in V_0$.
For any edge $f \in E$ and the corresponding cylinder set $[f]$, 
we set 
\begin{equation}\label{eq:nu([f])}  
    \nu([f]) = q_{r(f)}p_f. 
\end{equation}
Check that this definition of $\nu$ satisfies \eqref{eq_meas nu} 
for the cylinder sets $[f]$. Indeed, since $\tau_e^{-1}$ is defined 
only on $R_e = [e]$ and $\tau_f^{-1}([f]) = [r(f)]$, we have 
$$
\sum_{e \in E} p_e \nu\circ\tau_e^{-1}([f])  = p_f 
\nu\circ\tau_f^{-1}([f]) = p_f \nu([r(f)]) = q_{r(f)}p_f 
=  \nu([f]).
$$
Then, by induction, we define the values of $\nu$ on all cylinder 
sets  of length $n$:
$$
\nu([f_0, ..., f_{n-1}]) := p_{f_0}\cdots p_{f_{n-1}} 
q_{r(f_{n-1})}. 
$$
Verify that this definition satisfies the Kolmogorov extension 
theorem. Because
$$
[f_0, ..., f_{n-1}] = \bigcup_{e : s(e) = r(f_{n-1})} 
[f_0, ..., f_{n-1}, e],
$$
we find that 
$$
\ba
\sum_{e : s(e) = r(f_{n-1})} \nu([f_0, ..., f_{n-1}, e]) = &
\sum_{e : s(e) = r(f_{n-1})} p_{f_0}\cdots p_{f_{n-1}}p_e q_{r(e)}\\
= & p_{f_0}\cdots p_{f_{n-1}} \sum_{e : s(e) = r(f_{n-1})} 
p_e q_{r(e)}\\
= & p_{f_0}\cdots p_{f_{n-1}} q_{r(f_{n-1})}\\
= & \nu([f_0, ..., f_{n-1}])
\ea
$$
(we used here \eqref{eq: q =qP}). 

It remains to show that \eqref{eq_meas nu} holds for
any cylinder $[f_0, ..., f_{n}]$. Indeed,
$$
\ba
 \sum_{e \in E} p_e \nu\circ\tau_e^{-1}([f_0, ..., f_{n}]) =& 
 p_{f_0} \nu([f_1, ..., f_{n}]) \\
 =& p_{f_0} p_{f_1}\cdots p_{f_{n}} q_{r(f_{n})}  \\
=& \nu([f_0, ..., f_{n}]).
\ea
$$
Remark that the condition $P q = q$ is used to check that the 
measure $\nu$ defined inductively on cylinder sets can be extended 
to all Borel sets.
\end{proof}

\begin{theorem}\label{thm-IFS-inv}
Let $B$ be an irreducible generalized stationary Bratteli diagram, 
and let $\nu$ be the IFS 
measure defined in Theorem \ref{thm IFS stat BD}. Then:

(i) $\nu$ is not tail invariant;

(ii) $\nu$ is shift-invariant.
\end{theorem}

\begin{proof}
(i) Suppose that $\nu$ is tail invariant. Let $f,e \in E$ be two 
edges with $r(f) = r(e)$, i.e., $f, e$ are tail-equivalent. 
Then it follows from \eqref{eq:nu([f])} that $\nu([f]) = \nu([e])$ 
or
$$
q_{r(f)} p_f = q_{r(e)} p_e.
$$
This means that $p_f = p_e$. In other words, the matrix 
$$
P = (p_{w,v}) : e = (w, v)\in E)
$$
has constant columns. 

Consider two cylinder sets
$[f_0, f_1]$ and $[e_0, e_1]$ such that $r(f_1) = r(e_1)$ and 
there exist edges $g, h$ satisfying the properties: 
$s(g) = s(h)$, $r(g) = r_{f_0}$, and  $r(h) = r_{e_0}$. 
From the tail invariance of $\nu$ and the case considered above, 
we obtain that $p_{f_1} = p_{e_1}$ and 
$$
q_{r(f_1)} p_{f_0}p_{f_1} = q_{r(e_1)} p_{e_0}p_{e_1}
$$
Hence, $p_{f_0} = p_{e_0}$ and therefore $p_g = p_h$. This proves
that the rows of the matrix $P$ are constant. But the vector 
$ (p_e : e \in E)$ is probability, contradiction.
\\

(ii) Let $\varphi(x)$ be a bounded positive Borel function on the 
path space $X_B$. For the left shift $\sigma$ on $X_B$, compute 

\be 
\ba
\int_{X_B} \varphi(x) \; d\nu\circ\sigma^{-1}(x) = & 
\int_{X_B} \varphi(\sigma x) \; d\nu(x) \\
=& \sum_{e\in E} \int_{X_B}  p_e\varphi(\sigma x) \; d\nu\circ 
\tau_e^{-1}(x) \\
= & \sum_{e\in E} \int_{X_B}  p_e\varphi(\sigma (\tau_e x)) \; 
d\nu(x) \\
= & \sum_{e\in E} \int_{X_B}  p_e\varphi(x) \; d\nu(x) \\
= & \int_{X_B} \varphi(x) \; d\nu(x).
\ea
\ee
This proves that $\nu\circ \sigma^{-1} = \nu$. 
\end{proof}

\section{Measurable Bratteli diagrams and IFS measures} 
\label{sect MBD}

In this section, we discuss a measurable analogue of generalized 
Bratteli diagrams. The principal differences between 
discrete and measurable generalized Bratteli diagrams are: 
(a) the levels 
$V_n$ of a measurable Bratteli diagram are formed by standard Borel 
spaces $(X_n, \A_n)$,  and (b)  the sets of edges $E_n$ are Borel 
subsets of $X_n \times X_{n+1}$. Because these objects are 
non-discrete, we need to use new methods and techniques. 

\subsection{Measurable Bratteli diagrams and path space measures}
\label{ss MBD-IFS}
 
We give  the definitions of main objects in this subsection.  
 
\begin{definition}\label{def meas BD} 
(\textit{Measurable Bratteli diagrams})
Let $\{(X_n, \A_n) : n  \in \N_0\}$ be a sequence of standard Borel 
spaces. Let  $\{E_n : n \in N_0\}$ be a
sequence of Borel subsets such that $ E_n \subset 
X_n \times  X_{n+1}$. Denote by $\mathcal E = \bigsqcup_{i\geq 0}
E_i$ and $\mathcal X = \bigsqcup_{i\geq 0} X_i$ the sets of 
``edges'' and ``vertices'', respectively. 
For $e = (x,y) \in E_i$, the maps $s_i(x,y) = x$ and $r_i(x,y) =y$
are 
onto projections of $E_i$ to $X_i$ and $X_{i+1}$. Define 
$s, r$ on $\E$  by setting $s = s_i, r = r_i$ on $E_i$.
Then we call $\B = (\mathcal X, \E)$ a \textit{measurable Bratteli 
diagram}. The pair $(X_n, E_n)$ is called the $n$-th level of the   
measurable Bratteli diagram $\B$. 

If all $E_n = E$, then the measurable Bratteli diagram $\B$ is 
called \textit{stationary}.
 \end{definition} 
 
\begin{remark}\label{rem MBD}
(1) We  will identify the standard Borel spaces $\{(X_n, \A_n) : 
n  \in \N_0\}$ with an uncountable standard Borel space $(X, \A)$.
This means that a paint $x$ can be seen in all levels 
$(X_n, \A_n)$. This fact explains why we do not require that all
levels $X_n = X_0$ in the definition of a stationary Bratteli 
diagram. Nevertheless, we will keep using subindeces to indicate 
the level of a measurable Bratteli diagram. 

(2) The set of edges $E_i, i \geq 0,$ can be represented as 
follows:
$$
E_i = \bigcup_{x\in X_i} s^{-1}(x) = \bigcup_{x\in X_{i+1}} 
r^{-1}(x).
$$
This means that $E_i$ can be seen as the union of ``vertical'' and
``horizontal'' sections. 
The fact that $r$ is an onto map says that $\forall y \in X_{i+1}
\ \exists x \in X_i \ \mbox{such\ that} \ (x, y) \in E_i$. 
A similar property holds for the map $s$.

(3) To define the \textit{path space} $\mathcal X_{\B}$ of a 
measurable Bratteli diagram $\B = (\mathcal X, \E)$, 
we take a sequence $\ol x = \{ e_i = (x_i, x_{i+1})\}, e_i 
\in  E_i$, such that $r(e_i) = s(e_{i+1})$ for all $i$.
Equivalently, $\ol x = (x_0, x_1, x_2, ...)$ where every pair
$(x_i, x_{i+1})$ is in $E_i$, $i\in \N_0$. Then $s(\ol x) = x_0$.

(4) We denote 
\be\label{eq_paths from w}
\mathcal X_\B(w) = \{ \ol x \in \mathcal X_\B : s(\ol x) = w\}.
\ee
Clearly, the sets  $\mathcal X_\B(w)$ form a partition $\eta$ of 
$\mathcal X_\B$, and $X_0$ is the quotient space with respect to 
this partition. 
\end{remark}

We will consider below \textit{ measurable stationary Bratteli 
diagrams},  i.e., $E_i =E$. For every $e = (w,v)\in E$, we set 
$$
D_e = \{\ol x \in \mathcal X_{\B} : s(\ol x) = r(e) = v\},
$$
$$
R_e = \{ \ol x \in \mathcal X_{\B} : (x_0, x_1) = e\}.
$$
For $\ol x = (x_0, x_1, ...) \in D_e$, define the map $\tau_e : D_e
\to R_e$ by setting
$$
\tau_e(\ol x) = (w, v, x_1, x_2, ...).
$$
We note that $\ol x \in D_e$ means that $x_0 = v$ so that 
$\tau_e(\ol x)$ is well defined and belongs to $R_e$. Moreover, the
map $\tau_e$ is one-to-one on its domain and 
the map $\tau_e^{-1} : R_e \to D_e$ is defined by
$$
\tau_e^{-1}(e, \ol y) = \ol y.  
$$

We can consider the metric $\mbox{dist}$ on the path space 
$\mathcal X_\B$
similarly to the case of generalized Bratteli diagrams. As shown
in Remark \ref{rem contractive}, the maps $\tau_e$ are contractive
for all $e\in E$.

Let $\B$ be a measurable Bratteli diagram with the path space 
$\mathcal X_{\B}$. By a measure $\mu$ on $X_\B$, we mean a Borel 
positive (finite or sigma-finite) measure. 

Consider a sequence of Borel sets  $C_i \subset X_i : i\in \N_0$. 
Then 
$(C_0 \times C_1 \times\cdots \times  C_N) \cap \mathcal X_\B$  is 
called a \textit{cylinder subset} of $X_\B$ of length $N$ and
denoted by $[C_0, ... , C_N]$. In particular, we can consider the
set $R_e$ as a cylinder set $[e]$ defined by the edge 
$e$, $e\in E$. It defines the  partition 
\be\label{eq xi}
\xi = \{[e] : e \in E\}
\ee 
of $\mathcal X_\B$. We remark that the quotient 
$\mathcal X_\B/\xi$ is isomorphic to the set $E$. 
Since $D_e = \mathcal X_\B(r(e))$ for every $e \in E$, 
the partition $\eta$ 
(see Remark \ref{rem MBD} (4)) is an enlargement of $\xi$ because
$$
D_e = \bigcup_{f : s(f) = r(e)} [f].
$$
\vskip 3mm

\textbf{Observation}. The collection of cylinder sets 
$[C_0, ... , C_N]$ of any finite 
length generates the Borel sigma-algebra of $X_\B$. A measure 
$\mu$ on $X_\B$ is completely determined by its values on cylinder 
sets. These facts are obvious.
\vskip 3mm

\subsection{IFS measures on measurable Bratteli diagrams}

Let $\B = (\mathcal X, \E)$ be a stationary measurable Bratteli 
diagram. Recall that $E$ denotes a Borel subset of $X_0 \times 
X_1$ (in fact, $E$ can be viewed as a  subset of $X_0\times X_0$). 
Suppose that $p$ is a Borel measure on $E$, then $(E,p)$ becomes
a standard measure space. We consider
both cases, finite and sigma-finite measures $p$.

The key tool of this subsection is the disintegration theorem 
for a finite or sigma-finite measure with respect to a measurable
partition of a measure space. The literature on this subject is 
very extensive. We refer to the original paper by 
Rokhlin \cite{Rohlin1949} and more recent paper 
\cite{Simmons2012} (see also \cite{BezuglyiJorgensen2018(book)}).

\begin{definition} \label{def cond meas} 
Suppose that $(Z, \mc C, \mu)$ and $(Y, \mc D, \nu)$ are standard 
$\sigma$-finite measure spaces. Let $\pi : Z \to Y$ be a 
measurable function. A system of \textit{conditional measures} 
for $\mu$ with respect to $\pi$ is a collection of measures 
$\{\mu_y : y \in Y\}$ such that \\
(i) $\mu_y$ is a Borel measure on $\pi^{-1}(y)$,\\
(ii) for every $B \in \mc C$, the function $y\mapsto \mu_y(B)$ is 
measurable and $\mu(B) = \int_Y \mu_y(B) \; d\nu(y)$.
\end{definition}

\begin{theorem}[\cite{Simmons2012}] \label{thm Simmons}
Let $(Z, \mc C, \mu)$ and $(Y, \mc D, \nu)$ be as above. 
For any measurable function  $\pi : Z \to Y$ 
such that $\mu\circ \pi^{-1} \ll \nu$ there exists a uniquely 
determined system of conditional measures $(\mu_y)_{y \in Y}$ which 
disintegrates the measure $\mu$, i.e. 
$$
\mu = \int_{Y} \mu_y d\nu(y).
$$
\end{theorem}

\begin{remark} \label{rem meas p}
For the edge set $E$ of a stationary measurable Bratteli diagram 
$\B$ and a Borel measure $p$ on $E$, consider the partition of 
$E$ into the sets $r(s^{-1}(x)), x\in X_0$. This partition 
is measurable in the sense of \cite{Rohlin1949}. We can apply
the disintegration theorem cited above. More precisely, the
sets $r^{-1}(x) = \{(x, y) \in E : y \in X_1\}$, where $x\in X_0$,
are vertical sections of the set $E$. Setting $\wh p (\cdot) =
p(s^{-1}(\cdot))$, we have the projection of the measure $p$
onto $X_0$. Denoting by $x \mapsto p(x, \cdot)$ the corresponding 
system of conditional measures, we have 
$$
p= \int_{X_0} p(x, dy)\; d\wh p(x)
$$
where the measure $p(x, dy)$ is supported by the set $r(s^{-1}(x))$.
\end{remark}

We will consider several versions of disintegration theorem 
related to measures on the path space $\mathcal 
X_\B$ of a measurable Bratteli diagram $\B$. We assume that the
measures will satisfy conditions (i)  and (ii) formulated below.
\vskip 0.3cm

(i) For a given measure $m$ on $\mathcal X_\B$ and the set
$\mathcal X_\B(y)$ defined in \eqref{eq_paths from w}, 
the measurable function 
$$
q_m(y) = m(\mathcal X_\B(y)), \ y \in X_0,
$$
takes finite values. 

(ii) The partition $\xi$ defined in \eqref{eq xi} is obviously 
measurable, and the measure $m$ on $\mathcal X_\B$  can 
be disintegrated with respect to $\xi$. Denote by  $\wh m$ the
projection of $m$ onto $E = X/\xi$. Let $p$ be a measure on $E$.  
Assuming that $\wh m \ll p$ and applying Theorem \ref{thm Simmons},
we have 
\be\label{eq:dis}
m = \int_{E} m_{e} \; d p(e),
\ee
where  $e \mapsto m_e$ is the system of conditional measures of
$m$ with respect to $(E, p)$. 
We note that the conditional measure $m_e$ is supported by the set 
$[e]$, $e \in E$.
\vskip 0.3cm

The class of IFS measures is determined by a special form of 
measures $m_e$ defined in \eqref{eq:dis}. We will assume that 
the  measures considered on the path space $\mathcal X_\B$ satisfy
conditions (i) and (ii).

\begin{definition} (\textit{IFS measures and disintegration})
Let $\B$ be a stationary measurable Bratteli diagram, and
let $\{\tau_e, e \in E\}$ be the system of contractive maps 
defined in
subsection \ref{ss MBD-IFS}. Suppose $p$ is a fixed probability 
measure on
the set $E$. A Borel measure $\mu$ is called an \textit{IFS measure}
with respect to $p$ if 
\be\label{eq-IFS for MBD}
\mu = \int_E \mu\circ\tau_e^{-1}\; dp(e).
\ee
\end{definition}

Our goal is to find conditions under which a  
measure $m$ defined on the path space of a stationary measurable
Bratteli diagram is an IFS measure. 

\begin{theorem}\label{thm IFS MBD} Let $p$ be a Borel 
probability measure on
$E$ where $E$ is the edge set of a stationary measurable Bratteli
diagram $\B = (\mathcal X, \E)$. 
Let  $\mu$ be a measure on $\mathcal X_\B$ and $q(x) = 
\mu(\mathcal X_\B(x))$. 
Then $\mu$ is and IFS measure if and only if the following 
condition holds:
\be\label{eq P-MBD}
\int_{X_1} p(x, dy) q(y) = q(x).
\ee 
\end{theorem}

\begin{proof}
In the proof of the theorem, we apply the idea used in Theorem 
\ref{thm IFS stat BD}. 

We will construct an IFS measure $\mu$ by defining its values on 
cylinder sets of consequently increasing length. This construction
will be based on the application of \eqref{eq-IFS for MBD} so that
the measure $\mu$ will be an IFS measure automatically. 
Relation \eqref{eq P-MBD} is used to satisfy the Kolmogorov
extension theorem. 

As mentioned in Remark \ref{rem meas p},
the projection of $p$ onto $X_0$ defines the measure $\wh p$.
This allows us to define $\mu$ on cylinder subsets of 
$\mathcal X_\B$ generated by Borel sets on $X_0$:
$$
\mu([C_0]) = \int_{C_0} q(x_0) \; d\wh p(x_0), 
$$
where $C_0$ is a Borel subset of $X_0$.

For a cylinder set $[C_0, C_1] =\{ x= (x_i) \in \mathcal X_\B :
x_0 \in C_0, x_1 \in C_1\}$ of length two, we define
\be\label{eq_cyl set 2}
\mu([C_0, C_1]) = \int_{C_0} \left(\int_{C_1 \cap r(s^{-1}(x_0))}
p(x_0, dy) q(y) \right) \; d\wh p(x_0).
\ee

Show that this definition of the measure $\mu$ satisfies 
\eqref{eq-IFS for MBD}:
$$
\ba
 \int_E \mu\circ \tau_e^{-1} ([C_0, C_1]) \; 
d p(e)
= & \int_E \mu\circ \tau_e^{-1} ([C_0, C_1]\cap [e]) \; d p(e)\\
=& \int_{C_0} \int_{C_1 \cap r(s^{-1}(x_0))} \mu(\mathcal X_\B(y)\;
p(x_0, dy)\; d\wh p(x_0)\\
= & \int_{C_0} \int_{C_1 \cap r(s^{-1}(x_0))} q(y) \;
p(x_0, dy)\; d\wh p(x_0) \\
= & \mu([C_0, C_1])
\ea
$$

In general, if we defined the measure $\mu$ on cylinder sets 
of the form $[C_0, ..., C_k], k = 1,..., n-1,$, then we set
$$
\mu ([C_0, ..., C_n]) = \int_E \mu\circ \tau_e^{-1} 
([C_0, ..., C_{n-1}]) \; dp(e).
$$
This means that this definition of $\mu$ shows that relation
\eqref{eq-IFS for MBD} holds automatically. 

It remains to show that the measure $\mu$ is well defined, that is
it satisfies the Kolmogorov extension theorem. We check this
property for the cylinder sets of length two. Indeed, in this case 
we take $C_1 = X_1$ and compute
$$
\ba
\mu([C_0, X_1]) = & \int_{C_0} \left( \int_{X_1 \cap r(s^{-1}(x_0))}
q(y) p(x_0, dy) \right) d\wh p(y)\\
= & \int_{C_0} \left( \int_{r(s^{-1}(x_0))}
q(y) p(x_0, dy) \right) d\wh p(y)\\
=& \int_{C_0} q(x_0) \; d\wh p(x_0)\\
=& \mu([C_0]).
\ea
$$
The general case is proved analogously. 
\end{proof}

\begin{remark}
We note that the existence of an IFS measure on $\mathcal X_\B$
can be proved by using a fixed point theorem following 
\cite{Hutchinson1981}. We consider the metric $\mathrm{dist}$ on
$\mathcal X_\B$ defined in Remark \ref{rem prop BD}. 
Let 
$$
\mbox{Lip}_1 = \{ f : \mathcal X_\B \to \mathbb R : 
|f(x) - f(y) \leq \ \mbox{dist}(x, y)\}.
$$
For a measure $\nu$ on $\mathcal X_\B$, set 
$$
L(\nu) = \int_E \nu\circ\tau_e^{-1} \; dp(e).
$$
Since $\tau_e$ is a contractive map for every $e$, one can show that
$\rho(L(\nu), L(\mu)) \leq \rho (\nu, \mu)$ where the metric
$\rho$ is defined by 
$$
\rho (\nu, \mu) = \sup \Bigl\{ \big|  \int f \; d\nu - \int f\; 
d\mu\; \big| : f \in \mathrm{Lip}_1\Bigr\}.
$$
Then we conclude that there exists a measure $\mu_0$ such that
$L(\mu_0) = \mu_0$.
\end{remark}

\begin{remark}
The IFS measure $\mu$ defined in Theorem \ref{thm IFS MBD} is 
shift invariant. The proof of this fact is similar to that 
giving in Theorem \ref{thm-IFS-inv}. 
\end{remark}

\textbf{Acknowledgements}.  The authors are pleased to thank our 
colleagues and collaborators, especially, R. Curto, H. Karpel, 
P. Muhly, W. Polyzou, S. Sanadhya. We are 
thankful to the members of the seminars in Mathematical Physics and 
Operator Theory at the University of
Iowa for many helpful conversations. 

\textbf{Declaration}.
The authors declare that they have no conflict of interest.

\bibliographystyle{alpha}
\bibliography{IFS}
\end{document}